\documentclass[12pt,a4paper]{article}
\usepackage{amsmath,amssymb,amsthm}
\usepackage[english]{babel}
\usepackage[latin1]{inputenc}
\usepackage[T1]{fontenc}
\usepackage{times}
\usepackage{color,longtable}
\usepackage[title]{appendix}
\usepackage[normalem]{ulem}

\newtheorem{theorem}{Theorem}[section]

\theoremstyle{definition}
\newtheorem{definition}[theorem]{Definition}

\def\skipaline{\removelastskip\vskip10pt plus 1pt minus 1pt}
\def\skiphalfaline{\removelastskip\vskip5pt plus 1pt minus 1pt}

\newcommand{\LL}{\mathbb{L}\hskip0.2mm}

\newcommand{\R}{\mathbb{R}\hskip0.2mm}
\newcommand{\N} {\mathbb{N}\hskip0.2mm}

\DeclareMathOperator{\dd}{\mbox{\rm d}}

\newcommand{\J}{[0,2\pi]}
\newcommand{\I}{(0,\infty)}
\newcommand{\eps}{\varepsilon}
\newcommand{\wt}{\widetilde}
\newcommand{\infess}{\mbox{inf\ ess}}

\newcommand{\pruh}{\overline}

\numberwithin{equation}{section}

\begin{document}

\title{\sc  Multiple solutions of singular periodic problems}
\author{Milan Tvrd\'y, Institute of Mathematics, Czech Academy of Sciences\\[4mm]
{\small(based on joint research with {\bf Irena Rach\r{u}nkov\'a} and {\bf Ivo Vrko\v{c}})}}
\date{}
\maketitle

In this contribution we consider the singular periodic problem
\[
    u''+k\,u=g(u)+e(t),\quad u(0)=u(2\pi),\,\,u'(0)=u'(2\pi),
\]
where $k\in\R,$ $e:\J\to\R$ is Lebesgue integrable and $g:\I\mapsto\R$ is continuous,
but can have a singularity at $x=0.$
In particular, I wish to revisit our earlier result regarding the existence of
at least two its solutions and present a proof that is somewhat more transparent
than the original one.

Most of the symbols (e.g., for spaces, norms, etc.) appearing in this paper are generally accepted;
therefore, I will explain their meaning only in several exceptional cases.

\section{Method of lower and upper functions}

Our main tools will be the lower and upper functions associated to the given problem.
We will work with somewhat generalized versions of this concept.

\begin{definition}\label{LUF}
   Let $J$ be a subinterval of $\R$ and let $f$ be a Carath\'{e}odory function on
   $\J\times J.$ Then, a couple of functions $(\sigma,\rho)\in AC\J\times BV\J$
   \ are said to be {\em lower} functions of the problem
   \begin{equation}\label{gen}
      u''= f(t,u),\quad u(0)=u(2\pi),\,\, u'(0)=u'(2\pi)
   \end{equation}
   if\,\footnote{As usual, $\rho^{\mbox{\tiny{sing}}}$ \ stands for the singular part
   of\ $\rho\in BV\J.$}
   \vskip-10mm
   \begin{align*}
        &\sigma(t)\in J \mbox{ \ for a.e. \ } t\in\J,
        \quad \rho^{\mbox{\tiny{sing}}}\mbox{ \ is nondecreasing on \ } \J,
      \\[1mm]
        &\sigma'(t)=\rho(t) \mbox{ \ a.e. on \ } \J,
      \\[1mm]
        &\rho'(t)\ge f(t,\sigma(t)),\quad
         \sigma(0)=\sigma(2\pi),\,\,\rho(0+)\ge\rho(2\pi-).
  \end{align*}

  Similarly, a couple of functions $(\sigma,\rho)\in AC\J\times BV\J$ \ are {\em upper}
  functions of \eqref{gen} \ if \ $\rho^{\mbox{\tiny{sing}}}$ \ is nonincreasing and the
  inequalities above are reversed.
\end{definition}

The above notion of lower and upper functions is justified by the following theorem
which is contained in \cite[Theorems 4.1, 4.2 and 4.3]{RT2}.

\skipaline

\begin{theorem}\label{LUFTh}
   Let $J$ be a subinterval of $\R$ and let $f$ be a Carath\'{e}odory function on $\J\times J.$
   Further, let \ $(\sigma_1,\rho_1)$ and $(\sigma_2,\rho_2)$ \ be respectively lower and upper functions
   of the problem \eqref{gen}. Furthermore, assume that there is $m\in\LL\J$ such that $f(t,x)\ge m(t)$
   for a.e. $t\in\J$ and all $x\in\R$ $($or $f(t,x)\le m(t)$ for a.e. $t\in\J$ and all $x\in\R).$
   Then \eqref{gen} has a~solution \ $u$ \ such that $\|u'\|_{\infty}\le\|m\|_1.$ Moreover, if
   \begin{equation}\label{well}
       \sigma_1(t)\le\sigma_2(t)\quad\mbox{for all\ } t\in\J,
   \end{equation}
   then $\sigma_1(t)\le x(t)\le\sigma_2(t)$ is true for all $t\in\J$ and if \eqref{well} does not
   hold, then there is $t_u\in\J$ such that $\sigma_2(t_u)\le u(t_u)\le\sigma_1(t_u).$
\end{theorem}

\section{Main result}

Now, let us go back to the problem of our main interest, namely
\begin{equation}\label{P}
   u''+k\,u=g(u)+e(t),\,\, u(0)=u(2\pi),\,\, u'(0)=u'(2\pi).
\end{equation}

In some older papers, cf. e.g. \cite[Theorem 3.3]{RTV2}, we presented related existence results.
Next theorem ensuring the existence of at least two solutions to \ \eqref{P} \ has been
presented first in \cite{RTV2}, cf. Theorem 4.3 therein.

\begin{theorem}\label{Main}
   Assume:
   $g\in C(0,\infty),$ \ $k\in\R,$ \ $e\in L[0,2\pi]$ \ and
   \footnote{As usual, $\overline{e}$ stands for the integral mean value of $e,$ i.e.
   \ $\overline{e}=\frac{1}{2\pi}\int_0^{2\pi}e(s)\,\dd s.$}
   \begin{align}\label{1}
     &\infess_{t\in\J}\,e(t)>-\infty,
    \\[1mm]\label{2}
     &\lim_{x\to 0+}\,\int_x^1g(s)\,\dd s=+\infty,
    \\[1mm]\label{3}
     &\liminf_{x\to 0+}\,g(x)>-\infty,
    \\\label{4}
     &g(x)-k\,x<-\pruh{e} \mbox{ \ for \ } x\,{\in}\,[A,B\,],
          \mbox{\ where \ } A>0, \,\,
           B-A\ge\frac{\pi}{3}\,\|e-\overline{e}\|_1,
    \\[1mm]\label{5}
     &\mbox{the problem \ \eqref{P} \ has upper functions\ } (\sigma_2,\rho_2)
     \mbox{ \ such that\ } \ \sigma_2>B \mbox{ \ on \ } \J.
   \end{align}

   Then the problem \ \eqref{P} \ has at least two positive solutions.
\end{theorem}

My aim now is to present the proof of Theorem \ref{Main} in a form somehow more transparent
than that we did in \cite{RTV2}. Of course, as the available space is limited, I do have
to be a bit brief, after all.

\proof \ {\sc Step\,1}. \ Let \ $\gamma(t,s)$ \ be the Green's function for
\  $v''=0,\,\,v(0)=v(2\pi)=0,$
\[
   c_0=-\frac 1{2\pi}
   \int_0^{2\pi}\hskip-2mm\left(
                 \int_0^{2\pi}\hskip-3mm\gamma(t,s)\,(e(s)-\pruh{e})\,\dd s\hskip-1mm\right)
   \mbox{ \ and \ }
   \sigma_0(t)=c_0\,{+}\hskip-1mm\int_0^{2\pi}\hskip-3mm\gamma(t,s)\,(e(s)-\pruh{e})\,\dd s
   \mbox{ \ on \ } \J.
\]
Then \ $\sigma_0''(t)\,{=}\,e(t){-}\pruh{e}$ \ a.e. on $\J$,
\ $\sigma_0(0)\,{=}\,\sigma_0(2\pi)\,{=}\,c_{\,0}$, \ $\sigma_0'(0)\,{=}\,\sigma'_0(2\pi)$
\ and \ ${\pruh{\sigma}_0\,{=}\,0}$. Hence,\vskip-2mm
\[
    \|\sigma_0\|_\infty\le\sqrt{\frac{\pi}{6}}\,\|\sigma'_0\|_2,
\]\vskip-2mm\noindent
due to the Sobolev inequality. Consequently,
\[
   \|\sigma'_0\|_2^2\,{=}\left|\hskip-0.5mm
                              \int_0^{2\pi}\hskip-2mm\sigma''_0(t)\sigma_0(t)\,\dd t\right|
   {=}\left|\hskip-0.5mm\int_0^{2\pi}\hskip-2mm(e(t){-}\pruh{e})\,\sigma_0(t)\,\dd t\right|
   {\le}\|e-\pruh{e}\|_1\,\|\sigma_0\|_\infty
   {\le}\|e-\pruh{e}\|_1\,\sqrt{\frac{\pi}{6}}\,\|\sigma'_0\|_2,
\]\vskip-2mm\noindent
i.e. $\displaystyle \|\sigma'_0\|_2\le\sqrt{\frac{\pi}{6}}\,\|e-\pruh{e}\|_1.$
Now, using Sobolev inequality once more, we reach at\vskip-1mm
\begin{equation}\label{sob}
    \|\sigma_0\|_\infty\le\frac{\pi}{6}\,\|e-\pruh{e}\|_1.
\end{equation}\vskip-1mm
\noindent
Let \ $\displaystyle
\sigma_1(t)\,{=}\,A\,{+}\,\frac{\pi}{6}\,\|e{-}\pruh{e}\|_1\,{+}\,\sigma_0(t)$ for $t\in\J.$
Then $\sigma''_1(t)\,{=}\,e(t)\,{-}\,\pruh{e}$ \ a.e.,\ $\sigma_1(0)\,{=}\,\sigma_1(2\pi),$
$\sigma'_1(0)\,{=}\,\sigma'_1(2\pi)$ and, due to \eqref{4} and \eqref{sob},
$\sigma_1(t)\in [A,B]$ \ for $t\in\J$. \ Moreover, by \eqref{4},
\[
   \sigma''_1(t)+k\,\sigma_1(t)>e(t)-\pruh{e}+g(\sigma_1(t))\,{+}\,\pruh{e}
                              =e(t)+g(\sigma(t)) \quad\mbox{a.e. on \ } \J,
\]
and \ $(\sigma_1,\sigma'_1)$ \ are lower functions of \eqref{P}. In addition,
it is easy to see that there is a \ $\nu_0\in(0,A)$ \ such that \ $(\sigma_1-\nu,\sigma'_1)$
\ are lower functions of \ \eqref{P} \ for each $\nu\in[0,\nu_0],$ as well.

\skiphalfaline

Having in mind Theorem \ref{LUFTh}, we conclude that our problem \ \eqref{P} \ has
a solution \ $u$ \ such that \ $u(t)\in[\sigma_1(t),\sigma_2(t)]\subset[A,\sigma_2(t)]$
\ for \ $t\in\J.$

\skipaline

\noindent{\sc Step 2.} \quad
By \eqref{2} and \eqref{3}, there is $A_0\in (0,A-\nu)$ such that
\[
   g(A_0)-k\,A_0+e(t)\ge g(A_0)-k\,A_0+\infess_{t\in\J}\,e(t)\ge 0
   \quad\mbox{for a.e.\ } t\in\J.
\]
This means that $A_0$ is a constant upper function of \eqref{P}. Further, put
\[
    R\,{=}\,\|\sigma_2\|_\infty \quad\mbox{and}\quad
    m(t)\,{=}\,e(t)\,{+}\,\min\{0,\inf_{x\in (0,R]}\,(g(x)-k\,x\}
    \mbox{ \ for a.e. \ } t\in\J.
\]\vskip-3mm\noindent
Clearly, $m\in\LL\J.$ Denote\vskip-6mm
\[
   K^*=\|m\|_1\,\|e\|_1+\int_{A_0}^R|g(x)-k\,x|\,\dd x.
\]
Due to \eqref{2}, we can choose $\eps^*\in (0,A_0)$ in such a~way that
\ $g(\eps^*)-k\,\eps^*>0$ \ and
\begin{equation}\label{eps*}
    \int_{\eps^*}^{A_0}(g(x)-k\,x)\,\dd x>K^*.
\end{equation}
Now, put\vskip-5mm
\begin{align*}
  \wt{g}(x)
   &=\begin{cases}
            0 &\quad\mbox{if \ } x<0,
          \\
            g(\eps^*)\,\frac x{\eps^*}-k\,x &\quad\mbox{if \ } x\in(0,\eps^*),
         \\[1mm]
            g(x)-k\,x   &\quad\mbox{if \ } x\in[\eps^*,\sigma_2(t)]
         \end{cases}
   \\\noalign{\mbox{and}}
   \wt{f}(t,x)
   &=e(t)+\begin{cases}
           \wt{g}(x) \quad\mbox{if \ } x\in(0,\sigma_2(t)],
         \\
           g(x)-k\,x+\dfrac{x-\sigma_2(t)}{x-\sigma_2(t)+1}
           &\quad\mbox{if \ } x>\sigma_2(t)
        \end{cases}
\end{align*}
and consider a corresponding auxiliary problem
\begin{equation}\label{Ptilde}
     v''=\wt{f}(t,v),\,\, v(0)=v(2\pi),\,\, v'(0)=v'(2\pi).
\end{equation}
It is easy to verify that \ $\wt{f}(t,x)\ge m(t)$ for a.e. $t\in\J$ and all $x\in\R,$
while  $A_0$ is a constant upper function and $(\sigma_1-\nu,\sigma'_1)$ \ are lower
functions of \eqref{Ptilde}. Hence, the auxiliary problem \eqref{Ptilde} has a solution
$v$ \ such that \ $v(t_v)\in [A_0,\sigma_1(t_v)-\nu]$ \ for some \ $t_v\in\J.$ Since
$u(t)\in [\sigma_1(t),\sigma_2(t)]$ for all $t\in\J,$ it is $v\not\equiv u.$

It remains to show that \ $v(t)\in [\eps^*,\sigma_2(t))$ \ for all $t\in\J,$
which would mean that $v$ is a solution to our problem \eqref{P}, as well.

\skiphalfaline

\noindent{\sc Step 3.} \quad Assume that \
$\displaystyle\max_{t\in\J}(v(t)-\sigma_2(t))=(v(\tau)-\sigma_2(\tau))>0$ \ for
some \ $\tau\in(0,2\pi)$. Then \
$v'(\tau)\,{-}\,\rho_2(\tau-)\,{=}\,0\,{=}\,v'(\tau)\,{-}\,\rho_2(\tau+)$ and
there is a $\delta\,{>}\,0$ such that $v(t)\,{>}\,\sigma_2(t)$ on $[\tau,\tau+\delta].$
\\[3mm]
Thus,\vskip-9mm
\begin{align*}
   v''(t)&=e(t)-k\,\sigma_2(t)+g(\sigma_2(t))
   +\dfrac{v(t){-}\sigma_2(t)}{v(t){-}\sigma_2(t)+1}
  \\[0mm]
   &>e(t)-k\,\sigma_2(t)+g(\sigma_2(t))\ge \rho'_2(t)
    \quad \mbox{a.e.\ on\ } [\tau,\tau{+}\delta].
\end{align*}\vskip-2mm\noindent
Hence,\vskip-6mm
\[
   v'(t)-\rho_2(t)=\int_{\tau}^t(v''(s)-\rho'_2(s))\,\dd s>0
   \quad\mbox{\ for \ } t\in(\tau,\tau+\delta)
\]\vskip-3mm\noindent
and\vskip-8mm
\[
  v(\tau+\delta)-\sigma_2(\tau+\delta)
 =v(\tau)-\sigma_2(\tau)+\int_\tau^{\tau+\delta}(v'(s)-\rho_2(s))\,\dd s
 >v(\tau)-\sigma_2(\tau).
 \]
As this is a contradiction, it follows that \ $v(t)\le\sigma_2(t)$ \ on \ $\J.$
In particular,\vskip-3mm
\[
    \max_{t\in\J}v(t)=v(t_1)\in[A_0,R].
\]
Now, assume that \ $\min_{t\in\J}v(t)=v(t_0)<\eps^*.$ \ Due to the periodicity of $v,$
\ we have \vskip-2mm
\[
   v'(t_0)=v'(t_1)=0 \quad\mbox{and}\quad
   \int_{v(t_0)}^{v(t_1)}\wt{g}(x)\,\dd x+\int_{t_0}^{t_1}e(t)\,v'(t)\,\dd t=0.
\]
Indeed,
\begin{align*}
   \int_{v(t_0)}^{v(t_1)}\wt{g}(x)\,\dd x+\int_{t_0}^{t_1}e(t)\,v'(t)\,\dd t
  &=\int_{t_0}^{t_1}\wt{g}(v(t))\,v'(t)\,\dd t+\int_{t_0}^{t_1}e(t)\,v'(t)\,\dd t
 \\
  &=\int_{t_0}^{t_1}v''(t)\,v'(t)\,\dd t=\int_{v'(t_0)}^{v'(t_1)} y\,\dd y=0.
\end{align*}
Furthermore, by \eqref{eps*}
\begin{align*}
      K^*<&\int_{\eps^*}^{A_0}(g(x)-k\,\xi)\,\dd x
      \le\int_{v(t_0)}^{v(t_1)}(\wt{g}(x)\,d x
        -\int_{A_0}^{v(t_1)}(g(x)-k\,x)\,\dd x
    \\
     =&-\int_{t_0}^{t_1}e(t)\,v'(t)\,dt\,{-}\int_{A_0}^{v(t_1)}(g(x)-k\,x)\,\dd x
     \le\|e\|_1\,\|m\|_1+\int_{A_0}^{\|\sigma_2\|_\infty}\hskip-1mm|g(x)-k\,x|\,\dd x=K^*.
\end{align*}
This being a contradiction, we conclude that it must be \ $v(t)\ge\eps^*$ for all $t\in\J.$
As we already know that \ $v(t)\le\sigma_2(t)$ on $t\in\J,$ this completes the proof.

\section{Alternative results}
\begin{theorem}
Assume: \ $g\in C(0,\infty),$ \ $k\in\R,$ \ $e\in L[0,2\pi]$ \ and
\begin{align*}
   &\bullet\quad\infess_{t\in\J}\,e(t)>-\pruh{e},
  \\[1mm]
   &\bullet\quad\lim_{x\to 0+}\,\int_x^1g(s)\,\dd s=+\infty,
  \\
   &\bullet\quad g(x)-k\,x<-\pruh{e} \mbox{ \ for\ } x\,{\in}\,[A,B],
    \mbox{\ where \ } A>0 \mbox{\ and \ }
    B\,{-}\,A\,{\ge}\,\frac{2\pi}{6}\,\|e{-}\overline{e}\|_1,
  \\[1mm]
   &\bullet\quad\eqref{P} \mbox{ \ has upper functions \ } (\sigma_2,\rho_2) \mbox{ \ such that \ }
   \sigma_2>B \mbox{ \ on \ }\J.
\end{align*}
Then \ \eqref{P} \ has at least 2 positive solutions.
\end{theorem}

\skipaline

\noindent
For the next assertion, define
\[
   \Phi(k)=
   \left\{\begin{array}{ll}
       \min\left\{\dfrac\pi6,\dfrac{\coth(\sqrt{|k|})\,\pi}{4\sqrt{|k|}}\right\}
      &\quad\mbox{if}\,\,k<0,
   \\[6mm]
      \min\left\{\dfrac\pi{6\,(1-k)},\dfrac 1{4\,k}\sin(\sqrt{k}\,\pi)\right\}
     &\quad\mbox{if}\,\,k\in (0,\frac 14],
   \\[5mm]
      \min\left\{\dfrac{\pi}{6\,(1-k)},\dfrac 1{2\sqrt{k}\sin(\sqrt{k}\,\pi)}\right\}
     &\quad\mbox{if}\,\,k\in (\frac 14,1),
   \\[4mm]
     \dfrac 1{2\,\sqrt{k}\,|\sin(\sqrt{k}\,\pi)|}
    &\quad\mbox{if}\,\,k>1,\, k\ne n^2,\, n\in\N.
  \end{array}\right.
\]

\begin{theorem}
Assume: \ $g\in C(0,\infty),$ \ $k\in\R,$ \ $e\in L[0,2\pi],$ \ $k\ne n^2$ \ for all
\ $n\in\N$ \ and
\begin{align*}
   &\bullet\quad \infess_{t\in\J}\,e(t)>-\infty,
  \\
   &\bullet\quad \lim_{x\to 0+}\,\int_x^1g(s)\,ds=+\infty,\quad
                 \liminf_{x\to 0+}\,g(x)>-\infty,
  \\
   &\bullet\quad g(x)-k\,\dfrac{A+B}2<-\pruh{e} \mbox{\ for\ } x\in[A,B],
                 \mbox{\ where \ } B-A\ge 2\,\Phi(k)\,\|e-\pruh{e}\|_1,
  \\
  &\bullet\quad \eqref{P} \mbox{\ has upper functions \ }(\sigma_2,\rho_2)
                \mbox{\ such that\ }\sigma_2>B \mbox{\ on \ } \J.
\end{align*}

\noindent Then: \ \eqref{P} \ has at least 2 positive solutions.
\end{theorem}

The statements above entail a somewhat implicit requirement for the existence of "sufficiently large"
upper functions. Apart from obvious cases leading to the existence of constant upper functions with
the required properties, there are also not so trivial ways to ensure the existence of the needed
upper functions for problem \eqref{P}; see, e.g., \cite{RT3} or \cite{RT4}.

{\small

References to other related works can be found in the bibliographies of the cited papers,
particularly in \cite{RTV1} and \cite{RTV2}.

\end{document}